\documentclass[12pt]{amsart}
\usepackage{amscd, amsfonts, amssymb, mathrsfs}
\usepackage{fullpage}
\usepackage{latexsym}
\usepackage{verbatim}
\usepackage{enumerate}
\usepackage[all,cmtip]{xy}
\usepackage[colorlinks=true]{hyperref}
\usepackage{cite}

\numberwithin{equation}{section}

\newtheorem{thm}[equation]{Theorem}

\newtheorem{lem}[equation]{Lemma}
\newtheorem{cor}[equation]{Corollary}
\newtheorem{prop}[equation]{Proposition}

\theoremstyle{definition}

\newtheorem{ex}[equation]{Example}

\theoremstyle{remark}
\newtheorem{rem}[equation]{Remark}
\theoremstyle{remark}
\newtheorem{rems}[equation]{Remarks}

\DeclareMathOperator{\GL}{GL}
\DeclareMathOperator{\SL}{SL}
\DeclareMathOperator{\SP}{Sp}

\DeclareMathOperator{\Lie}{Lie}
\DeclareMathOperator{\Ad}{Ad}

\renewcommand{\Im}{{\rm Im}}

\newcommand{\ovl}{\overline}

\newcommand{\Gm}{\mathbb{G}_m}
\renewcommand{\AA}{\mathbb{A}}
\newcommand{\ZZ}{\mathbb{Z}}


\title[Overgroups]
{Overgroups of regular unipotent elements in reductive groups}

\dedicatory{Dedicated to Professor Jean-Pierre Serre on the occasion of his 95$^{\it th}$ birthday,\\ with great admiration.}

\author[M.\  Bate]{Michael Bate}
\address
{Department of Mathematics,
University of York,
York YO10 5DD,
United Kingdom}
\email{michael.bate@york.ac.uk}

\author[B.\ Martin]{Benjamin Martin}
\address
{Department of Mathematics,
University of Aberdeen,
King's College,
Fraser Noble Building,
Aberdeen AB24 3UE,
United Kingdom}
\email{b.martin@abdn.ac.uk}
 
\author[G. R\"ohrle]{Gerhard R\"ohrle}
\address
{Fakult\"at f\"ur Mathematik,
Ruhr-Universit\"at Bochum,
D-44780 Bochum, Germany}
\email{gerhard.roehrle@rub.de}

\begin{document}

\begin{abstract}
 We study reductive subgroups $H$ of a reductive linear algebraic group $G$ --- possibly non-connected --- such that $H$ contains a regular unipotent element of $G$.  We show that under suitable hypotheses, such subgroups are $G$-irreducible in the sense of Serre.  This generalizes results of Malle, Testerman and Zalesski.  We obtain analogous results for Lie algebras and for finite groups of Lie type.  Our proofs are short, conceptual and uniform.
\end{abstract}

\maketitle

\section{Introduction}
\label{sec:intro}

Much effort has gone into describing the subgroup structure of reductive algebraic groups.  In this paper we study reductive subgroups containing a regular unipotent element of the ambient group.  
For simple $G$, Saxl and Seitz determined the maximal closed positive-dimensional subgroups containing a regular unipotent element of $G$ in \cite{saxl-seitz}, building on work of Suprunenko \cite{suprunenko}. 
Subsequently, these classifications have been extended and refined, for example by Testerman-Zalesski \cite{TZ}, Guralnick-Malle \cite{GM}, and Craven \cite{craven}, so that there is now a very good understanding of how subgroups containing regular unipotent elements can arise ``in nature''.

Using these classification results, Testerman and Zalesski proved the following striking result in \cite[Thm.~1.2]{TZ}: if $G$ is connected and $H$ is a connected reductive subgroup of $G$ containing a regular unipotent element of $G$ then $H$ is $G$-irreducible in the sense defined by J-P. Serre (i.e., is not contained in any proper parabolic subgroup of $G$).  Note that this is false if we replace ``regular unipotent'' with ``regular semisimple'': just take $H$ to be a Levi subgroup of a proper parabolic subgroup of $G$; then $H$ is connected reductive, contains a maximal torus of $G$ (and hence regular semisimple elements of $G$), but $H$ is not $G$-irreducible.

Malle and Testerman extended this result to non-connected $H$ inside simple $G$ \cite[Thm.~1]{MT}, and also considered a few cases when $G$ is non-connected.  (The notion of a regular unipotent element of a non-connected reductive group $G$ was introduced by Spaltenstein: see Section~\ref{sec:regunipotent} below.)  The proofs of \cite[Thm.~1.2]{TZ} and \cite[Thm.~1]{MT} involved long and intricate case-by-case considerations for the various possible Dynkin types of $G$.  

The first purpose of this paper is to give a short and uniform proof of the following more general result.
For the definitions of $G$-complete reducibility and $G$-irreducibility, see Section~\ref{sec:Gcr}.

\begin{thm}
\label{thm:main}
Let $H \subseteq G$ be reductive groups (possibly non-connected). 
Suppose $H$ contains a regular unipotent element of $G$.
Then:
\begin{itemize}
\item[(i)] The identity component $H^0$ of $H$ is $G$-completely reducible.
\item[(ii)] If the projection of $H^0$ onto each simple factor of the identity component $G^0$ of $G$ is not a torus, then $H$ does not normalize any proper parabolic subgroup of $G^0$.  
\item[(iii)] If the hypothesis of (ii) holds and $H$ meets every connected component of $G$, then $H$ is $G$-irreducible.
\end{itemize}
\end{thm}

The key ingredient in our proof is the observation, due to Steinberg 
 (for connected $G$, \cite[Sec.~3.7, Thm.~1]{steinberg:conjugacy}) and Spaltenstein (for non-connected $G$, \cite[Prop.~II.10.2]{spaltenstein}), that a regular unipotent element normalizes a unique Borel subgroup of $G^0$.

\begin{rems}
	\label{rem:main}	
	(i). The conclusion of part (i) follows from \cite[Thm.~3.10]{BMR} whenever (ii) or (iii) holds, because $H^0$ is normal in $H$. Note, however, that (i) in fact holds in complete generality --- i.e., without the additional hypotheses of (ii) and (iii). For non-connected $H$, this is a new result even in the case that $G$ itself is connected.  
	See also Corollary \ref{cor:connectedcr} below.
		
	(ii). Note that the hypotheses in Theorem~\ref{thm:main}(ii) and (iii) are automatic if $H$ and $G$ are both connected: for if $H$ is connected and contains a regular unipotent element of $G$, 
	then $H$ cannot project to a torus in any simple factor of $G$, and if $G$ is connected, then $H$ meets every component of $G$. 
	Hence Theorem~\ref{thm:main} specializes to \cite[Thm.~1.2]{TZ} in this case.
	
	(iii). For $G$ simple, we recover \cite[Thm.~1]{MT}, and for $G^0$ simple we get 
	\cite[Cor.~6.2]{MT}.
		
	(iv). We note that the restriction on $H^0$ in (ii) and (iii) is necessary. 
	For let $G$ be connected in positive characteristic and let $H$ be the closed subgroup of $G$ generated by a regular unipotent element $u$ of $G$. 
	Then, since $u$ is contained in a unique Borel subgroup $B$ of $G$, \cite[Sec.~3.7, Thm.~1]{steinberg:conjugacy}, so is $H$, and so $H$ is not $G$-irreducible. For  instances of a positive-dimensional reductive subgroup $H$  containing a regular unipotent element of $G$ which is not $G$-irreducible, see 
	\cite[Sec.~7]{MT}.
	
	(v). We observed above that (ii) and (iii) can fail if we replace ``regular unipotent'' with ``regular semisimple''.  In fact, (i) can also fail: e.g., take $H$ to be the image of the adjoint representation of $\SL_2$ in $G= \SL_3$ in characteristic~$2$ (note that since $H$ does not act completely reducibly on the natural module for $G$, $H$ is not $G$-cr: see \cite[Sec.~1]{BMR:semisimplification}).
\end{rems}

Many of the technicalities in the proof of Theorem~\ref{thm:main} disappear in the special case where both $G$ and the reductive subgroup $H$ are connected.  We give a separate short proof in this case which uses only very basic properties of reductive groups and regular unipotent elements.  It illustrates some of the key ideas of the general case, and a slight variation gives an analogous result for Lie algebras when $\Lie(H)$ contains a regular nilpotent element of $\Lie(G)$ (see Theorem~\ref{thm:connected}).

Our second main result is an analogue of Theorem~\ref{thm:main} for finite groups of Lie type.
Suppose $G$ is a connected reductive group and recall that a \emph{Steinberg endomorphism} of $G$ is a surjective morphism $\sigma:G\to G$ such that the 
corresponding fixed point subgroup $G_\sigma :=\{g \in G \mid \sigma(g) = g\}$ of $G$ is finite; 
Frobenius endomorphisms of reductive groups over finite fields are familiar examples, giving rise to \emph{finite groups of Lie type}, see \cite[Sec.~10]{steinberg:end}. 
Let $\sigma$ be a Steinberg endomorphism of $G$ and
suppose $H$ is a connected reductive $\sigma$-stable subgroup of $G$. 
Then $\sigma$ is also a Steinberg endomorphism for $H$ with finite fixed point subgroup  $H_\sigma = H \cap G_\sigma$, \cite[7.1(b)]{steinberg:end}.
Obviously, one cannot immediately appeal to Theorem~\ref{thm:main} to deduce anything about $H_\sigma$, because $H_\sigma^0$ is trivial.
However, our proof does still go through with some minor changes. 
We give here the version of the result for connected groups; see Proposition~\ref{prop:steinbergendgeneral} for the most general analogue of Theorem~\ref{thm:main} in this setting.

\begin{thm}
	\label{thm:finite}	
	Let $H \subseteq G$ be connected reductive groups 
	and suppose $\sigma$ is a Steinberg endomorphism of $G$ such that $H$ is $\sigma$-stable. 
	Suppose that $H$ contains a regular unipotent element of $G$. Then 
	$H_\sigma$ is $G$-irreducible. 
\end{thm}

As is pointed out in \cite[Sec.~1]{BT}, there are instances where one can embed a finite group of Lie type into a connected reductive group 
$G$ so that the image contains a regular unipotent element of $G$ but is \emph{not} $G$-irreducible.
For example, ${\rm PSL}_2(p)$ has a $p$-dimensional reducible indecomposable representation $V$ such that the image contains an element acting as a single Jordan block on $V$, and hence the image in $\GL(V)$ contains a regular unipotent element but is not $G$-irreducible; see \cite[Sec.~2.1]{BT} and \cite[p48]{alperin}. 
Theorem~\ref{thm:finite} shows that such a finite subgroup cannot arise as the fixed point subgroup of a connected reductive $\sigma$-stable subgroup $H$ of $G$ (since a subgroup $M$ of $\GL_n$ is $\GL_n$-irreducible if and only if the corresponding representation of $M$ is irreducible).
This was proved for exceptional simple $G$ and subgroups isomorphic to ${\rm PSL}_2(p)$ by an exhaustive case check in \cite[Thm.~2]{BT}; 
our result holds for arbitrary reductive $G$ and finite subgroups of arbitrary Lie type.

\bigskip
The proofs of Theorems~\ref{thm:main} and \ref{thm:finite} use the machinery of 
$G$-complete reducibility and optimality developed by the authors and others in a series of papers \cite{cochar, BMR, BMR:commuting, BMR:semisimplification, BMRT, GIT}.  
This yields, for instance, a very quick way to see Theorem~\ref{thm:main} in characteristic 0 (Remark~\ref{rem:char0}).  These methods are particularly well-suited to dealing with non-connected $G$.  We prove Theorem~\ref{thm:main} in full generality in Section~\ref{sec:proof};  
the shorter argument for connected $G$ and $H$ is given in Section~\ref{sec:conn}.     

During the build-up to the main proof, we show that the notion of regular unipotent element behaves well when passing to quotients and reductive subgroups of $G$ (Section~\ref{sec:regunipotent}); we believe this is of independent interest.
We also give some natural
examples in Section~\ref{sec:complements} 
where $H^0$ is a torus --- so the hypotheses of Theorem~\ref{thm:main} fail --- but $H$ is still $G$-irreducible.

\section{Preliminaries}
\label{sec:prelims}

Throughout, we work over an algebraically closed field $k$ of characteristic $p\geq 0$. 
A linear algebraic group $H$ over $k$ has identity component $H^0$; if $H=H^0$, then we say that $H$ is \emph{connected}.
We denote by $R_u(H)$ the \emph{unipotent radical} of $H$; if $R_u(H)$ is trivial, then we say $H$ is \emph{reductive} --- 
we do not insist that a reductive group is connected.
The derived subgroup of $H$ is denoted by $[H,H]$, the centre of $H$ by $Z(H)$, and its Lie algebra by $\Lie(H)$.

Throughout, $G$ denotes a reductive linear algebraic group over $k$. 
The semisimple group $[G^0,G^0]$ can be written as a product $G_1\cdots G_r$ of pairwise commuting simple groups $G_1,\ldots, G_r$; these are the \emph{simple factors of $G^0$}.
For each $i$ there is a surjective homomorphism from $G^0$ onto a quotient of $G_i$ by a finite subgroup; we call this map \emph{projection of $G^0$ onto the $i$th simple factor}.  
Given any element $g\in G$, the $G^0$-conjugacy class of $g$ is denoted by $G^0 \cdot g$; the Zariski closure of this class is denoted $\ovl{G^0 \cdot g}$.

\subsection{Endomorphisms}
\label{sec:steinbergend}

We give two results of Steinberg \cite{steinberg:end} which are used in the sequel.

\begin{lem}
	\label{lem:fixedborel}
	Let $H$ be a linear algebraic group and let $\sigma:H\to H$ be any surjective homomorphism. Then $\sigma$ stabilizes a Borel subgroup of $H$.
	In particular, for every $x\in H$ there is a Borel subgroup of $H$ normalized by $x$.
\end{lem}

\begin{proof}
The first statement is precisely \cite[Thm.~7.2]{steinberg:end}. The second follows by applying this to the endomorphism given by conjugation by $x$.
\end{proof}

Recall that a \emph{Steinberg endomorphism} of a linear algebraic group $H$ is a surjective endomorphism $\sigma:H\to H$ such that the fixed point subgroup 
$H_\sigma$ is finite.
As noted in Section~\ref{sec:intro}, if $\sigma$ is a Steinberg endomorphism of $H$, then the restriction of $\sigma$ to $H^0$ is a Steinberg endomorphism of $H^0$.
Hence we may deduce the following by applying \cite[10.4, Cor.~10.10]{steinberg:end} to $H^0$.

\begin{lem}
\label{lem:fixedboreltorus}
Let $H$ be a linear algebraic group and $\sigma$ a Steinberg endomorphism of $H$.
\begin{itemize}
\item[(i)] Each $\sigma$-stable Borel subgroup of $H$ contains a $\sigma$-stable maximal torus.
\item[(ii)] Any two pairs consisting of a $\sigma$-stable Borel subgroup and a $\sigma$-stable maximal torus of $H$ are conjugate by an element of $(H^0)_\sigma$.
\end{itemize}
\end{lem}

\subsection{Cocharacters and R-parabolic subgroups}
For a linear algebraic group $H$, we let $Y(H)$ denote the set of cocharacters of $H$; that is, the set of algebraic group homomorphisms $\lambda:\Gm\to H$.
The group $H$ acts on the set of cocharacters: for $\lambda\in Y(H)$ and $h\in H$ we write $h\cdot\lambda$ for the cocharacter defined by $(h\cdot\lambda)(t) = h\lambda(t)h^{-1}$ for each $t\in \Gm$.
Given an affine variety $X$ and a morphic action of $H$ on $X$, for each $\lambda\in Y(H)$ and $x\in X$ we can define a morphism $\phi_{x,\lambda}:\Gm\to X$ by
the rule $\phi_\lambda(t) = \lambda(t)\cdot x$.
Identifying $\Gm$ as a principal open set in $\AA^1$ in the usual way, if $\phi_{x,\lambda}$ extends to a (necessarily unique) morphism $\widehat{\phi}_{x,\lambda}$ from all of $\AA^1$ to $X$, then we say that $\lim_{t\to 0}\lambda(t)\cdot x$ \emph{exists} and set $\lim_{t\to 0}\lambda(t)\cdot x= \widehat{\phi}_{x,\lambda}(0)$.

This set-up is important to us in this paper when we consider the action of $G$ on itself by conjugation.
Here, for each $\lambda\in Y(G)$, the set $P_\lambda:=\{g\in G\mid \lim_{t\to 0}\lambda(t)g\lambda(t)^{-1} \textrm{ exists}\}$ is
a so-called \emph{R-parabolic subgroup of $G$} \cite[Sec.~6]{BMR}.
An R-parabolic subgroup of $G$ is a parabolic subgroup of $G$ in the usual sense, and it has a \emph{Levi decomposition}
$P_\lambda = L_\lambda \ltimes R_u(P_\lambda)$, where
\begin{align*}
L_\lambda & :=\{g\in G\mid \lim_{t\to0} \lambda(t)g\lambda(t)^{-1}=g \} = C_G(\Im(\lambda)), \\
R_u(P_\lambda) &= \{g\in G\mid  \lim_{t\to 0}\lambda(t)g\lambda(t)^{-1}=1 \};
\end{align*}
see \cite[Prop.\ 5.2]{martin1} for this description of $R_u(P_\lambda)$.
Since $R_u(P_\lambda)$ is connected, $P_\lambda$ and $L_\lambda$ have the same number of connected components.  We call $L_\lambda$ an \emph{R-Levi subgroup} of $P_\lambda$.
Note that for all $g\in P_\lambda$, we have $\lim_{t\to0}\lambda(t)g\lambda(t)^{-1} \in L_\lambda$ --- in fact, the map $g\mapsto \lim_{t\to0}\lambda(t)g\lambda(t)^{-1}$ is the canonical
projection $P_\lambda \to L_\lambda$ which arises by quotienting out $R_u(P_\lambda)$.
For more properties of these subgroups, see \cite[Sec.~6]{BMR};
we recall here that for connected $G$, the R-parabolic subgroups and their R-Levi subgroups are precisely the parabolic subgroups and their Levi subgroups
\cite[Sec.~8.4]{springer}.  Moreover, $P_\lambda\cap P_{-\lambda}= L_{\lambda}$, so if $G$ is connected then $P_\lambda$ and $P_{-\lambda}$ are opposite parabolic subgroups.  

These results have 
analogues in the Lie algebra $\Lie(G) = \Lie(G^0)$ of $G$.
Recall that $G$ acts on $\Lie(G)$ via the adjoint representation $\Ad$, and then for each $\lambda\in Y(G)$ we have:
\begin{align*}
\Lie(P_\lambda) &= \{X\in \Lie(G)\mid \lim_{t\to0} \Ad(\lambda(t))(X) \textrm{ exists}\}\\
\Lie(L_\lambda) &= \{X\in \Lie(G)\mid \lim_{t\to0} \Ad(\lambda(t))(X) = X\}\\
\Lie(R_u(P_\lambda)) &= \{X\in \Lie(G)\mid \lim_{t\to0} \Ad(\lambda(t))(X) =0\};
\end{align*}
see, e.g., \cite[Sec.~2]{rich1}.

If $H$ is a reductive subgroup of $G$, then we may identify $Y(H)$ with a subset of $Y(G)$.
Then a cocharacter of $H$ gives rise to an R-parabolic subgroup of $H$ and of $G$ --- in this situation,
we write $P_\lambda(H)$ for the R-parabolic subgroup of $H$ and reserve the notation $P_\lambda$ for the R-parabolic subgroup of $G$;
we similarly write $L_\lambda(H)$.
It is clear from the definitions that $P_\lambda(H) = P_\lambda\cap H$, $L_\lambda(H) = L_\lambda\cap H$ and $R_u(P_\lambda(H)) = R_u(P_\lambda)\cap H$.

In what follows, we occasionally need to use the root system of $G^0$, so we introduce some notation here.
Let $T$ be a maximal torus of $G$ and let
$\Phi = \Phi(G^0,T)$ be the set of roots of $G^0$ with respect to $T$. 
Let $B$ be a Borel subgroup of $G$ containing $T$ and let 
$\Phi^+ = \Phi(B,T)$ denote the positive system of roots with respect to $B$. 
For each $\alpha\in \Phi$ we have a root subgroup $X_\alpha$ of $G$.
For a cocharacter $\lambda\in Y(T)$, we have $X_\alpha\subseteq P_\lambda$ 
if and only if $\langle\lambda,\alpha\rangle \geq 0$, where $\langle \, ,\, \rangle:Y(T)\times X(T)\to \ZZ$ is the usual pairing between cocharacters and characters of $T$.
We have $X_\alpha\subseteq L_\lambda$ if and only if $\langle\lambda,\alpha\rangle = 0$,
and also $R_u(P_\lambda)$ is generated by the $X_\alpha$ with $\langle\lambda,\alpha\rangle >0$; cf.\ the proof of \cite[Prop.~8.4.5]{springer}.

We finish this section with a key result \cite[Prop.\ 5.4(a)]{martin1} which we use often in the sequel (note that in \emph{loc.~cit.} R-parabolic subgroups are called ``generalized parabolic subgroups'').

\begin{lem}\label{lem:N_G(P)}
Suppose $P$ is a parabolic subgroup of $G^0$.
Then $N_G(P)$ is an R-parabolic subgroup of $G$ with $N_G(P)^0 = P$.
\end{lem}

\subsection{$G$-complete reducibility and optimal R-parabolic subgroups}
\label{sec:Gcr}

We collect some basic results concerning Serre's notion of complete reducibility;
for further background and results, see \cite{serre1.5}, \cite{serre2}, \cite{BMR}.
A subgroup $H$ of $G$ is called \emph{$G$-completely reducible} ($G$-cr) if whenever 
$H\subseteq P$ for an R-parabolic subgroup $P$, there exists an R-Levi subgroup $L$ of $P$ with $H\subseteq L$.
If $H$ is a subgroup of $G^0$, then $H$ is $G$-cr if and only if $H$ is $G^0$-cr \cite[Prop.~2.12]{BMR:commuting}.  Note that if $G^0$ is a torus then $R_u(P_\lambda)= 1$ for any $\lambda\in Y(G)$, so every subgroup of $G$ is $G$-cr.

A subgroup $H$ of $G$ is \emph{$G$-irreducible} ($G$-ir) if $H$ is not contained in any proper R-parabolic subgroup of $G$; 
a $G$-ir subgroup is automatically $G$-cr.
We note that if $H$ meets every component of $G$, then $H$ is $G$-ir if and only if $H$ normalizes no proper parabolic subgroup of $G^0$ --- this follows from Lemma~\ref{lem:N_G(P)}.
Only the forward implication holds if $H$ does not meet every component of $G$:
whenever $Z(G^0)$ is not central in $G$, there are cocharacters $\lambda\in Y(Z(G^0))$ such that
$P_\lambda = L_\lambda$ is a proper subgroup of $G$.
These subgroups are $G$-cr but not $G$-ir, and yet have identity component equal to $G^0$, so do not normalize any proper parabolic subgroup of $G^0$.

Our next result is an easy fact about $G$-complete reducibility which we use in the proof of part (i) of Theorem~\ref{thm:main}.

\begin{lem}\label{lem:derivedcr}
	Suppose $K$ is a connected reductive subgroup of $G$.
	Then $K$ is $G$-completely reducible if and only if $[K,K]$ is $G$-completely reducible.
\end{lem}

\begin{proof}
	We may write $K = [K,K]Z$, with $Z= Z(K)^0$. Let $L=C_G(Z)$.
	Since $Z$ is a torus centralizing $K$ and $[K,K]$, we have that $K$ (resp.~$[K,K]$) is $G$-cr if and only if $K$ (resp.~$[K,K]$) 
	is $L$-cr, by \cite[Cor.~3.22, Sec.~6.3]{BMR}. 
	But $Z$ is contained in every R-parabolic subgroup and every  R-Levi subgroup of $L$, because $Z$ is a central torus in $L$.
	So $K$ is $L$-cr if and only if $[K,K]$ is $L$-cr.
\end{proof}

The next result follows quickly from \cite[Prop.\ 4.11]{boreltits}. 
We give the details since they are useful in what follows.

\begin{lem}
	\label{lem:opp_rads}
	Let $P$ and $Q$ be opposite parabolic subgroups of $G^0$.  
	Let $M$ be the subgroup of $G$ generated by $R_u(P)\cup R_u(Q)$. 
	Then $M$ is connected and $G$-completely reducible.  
	Moreover, if $P$ and $Q$ do not contain any simple factors of $G^0$ then $M = [G^0,G^0]$.
\end{lem}

\begin{proof}
	Since $M$ is generated by the connected groups $R_u(P)$ and $R_u(Q)$, $M$ is connected by \cite[Prop.\ 2.2]{borel}.
	Now we use the proof of \cite[Prop.\ 4.11]{boreltits}: the opposite parabolic subgroups $P$ and $Q$ have a common Levi subgroup $L$ which normalizes $R_u(P)$ and $R_u(Q)$.
	Hence $N_{G^0}(M)$ contains $R_u(P)$, $R_u(Q)$ and $L$, which puts a maximal torus $T\subseteq L$ and all the root subgroups of $G^0$ inside $N_{G^0}(M)$.
	Thus $N_{G^0}(M) = G^0$, and we see that
	$M$ is normal in $G^0$ (which is the result of \emph{loc. cit.}). 
	Therefore $M$ is $G^0$-cr, by \cite[Thm.\ 3.10]{BMR}, and hence $M$ is $G$-cr.
	
	For the final assertion, let $G_1,\ldots, G_r$ be the simple factors of $G^0$.
	For $1\leq i \leq r$, let $P_i$ and $Q_i$ denote the opposite parabolic subgroups of $G_i$ corresponding to $P$ and $Q$, and let $M_i$ denote the subgroup of $G_i$ generated by $R_u(P_i)$ and $R_u(Q_i)$;
	by the first paragraph, $M_i$ is normal in $G_i$.  
	The hypothesis that $G_i$ is not contained in $P_i$ and $Q_i$ implies that $M_i$ is a positive-dimensional connected normal subgroup of $G_i$, and hence $M_i = G_i\subseteq M$.
	Thus the final part of the statement also holds.   
\end{proof}

If $H$ is a subgroup of $G$ which is not $G$-cr, then there is a way to associate to $H$ a so-called 
\emph{optimal R-parabolic subgroup} $P$ of $G$: see \cite[Sec.~4]{GIT}.

\begin{thm}\label{thm:optpar}
Suppose that the subgroup $H$ of $G$ is not $G$-completely reducible. 
Then there exists an R-parabolic subgroup $P$ of $G$ with the following properties:
\begin{itemize}
\item[(i)] $H$ is not contained in any R-Levi subgroup of $P$;
\item[(ii)] $N_G(H)\subseteq P$.
\end{itemize}
\end{thm}
The construction of $P$ relies on the geometric characterisation of complete reducibility introduced in \cite{BMR} and developed further in \cite{GIT} --- roughly speaking, one associates to $H$ an orbit in an affine $G$-variety, and 
then the R-parabolic subgroup arises from the \emph{optimal class} of cocharacters for that orbit; see also \cite{kempf}.

We finish the section with a result we need for the proof of Theorem~\ref{thm:main}.

\begin{lem}
\label{lem:torimage}
 Let $\pi\colon G\to G'$ be a homomorphism of connected reductive groups.  Let $\lambda\in Y(G)$ such that $P_\lambda$ is a Borel subgroup of $G$.  Suppose $\pi(G)$ is not a torus.  Then $\pi\circ \lambda$ is nontrivial.  In particular, if $G'$ is simple then $P_{\pi\circ \lambda}\subsetneq G'$.
\end{lem}

\begin{proof}
 Let $G_1,\ldots, G_r$ be the simple factors of $G$ and let $Z= Z(G)^0$.  Let $\mu\colon G_1\times\cdots\times G_r\times Z\to G$ be the multiplication map.  Since $\mu$ is an isogeny, there exist $n\in {\mathbb N}$ and $\nu\in Y(G_1\times\cdots\times G_r\times Z)$ such that $\mu\circ \nu= n\lambda$.  By \cite[Prop.~2.11]{BMR}, $P_\nu= \mu^{-1}(P_{n\lambda})= \mu^{-1}(P_\lambda)$, so $P_\lambda$ is a Borel subgroup of $G$ if and only if $P_\nu$ is a Borel subgroup of $G_1\times\cdots\times G_r\times Z$.  Without loss, therefore, we can assume that $G= G_1\times\cdots\times G_r\times Z$ and $\nu= \lambda$.
 
 Suppose $\pi\circ \lambda$ is trivial.  We can write $\lambda= \lambda_1 \times \cdots \times  \lambda_r \times \epsilon$, 
 where each $\lambda_i$ belongs to $Y(G_i)$ and $\epsilon$ belongs to $Y(Z)$.  Now ${\rm ker}(\pi)^0$ is the product of certain of the $G_i$ with a subtorus of $Z$.  Each $\lambda_i$ is nontrivial since $P_\lambda$ is a Borel subgroup, so ${\rm ker}(\pi)$ must contain $G_1\times\cdots\times G_r$.  The result follows.
\end{proof}

\section{The connected case}
\label{sec:conn}

Recall that if $G$ is connected then $g\in G$ is {\em regular} if $\dim(C_G(g))$ is minimal.
We need two properties of regular unipotent and nilpotent elements for connected reductive groups.

\begin{lem}\label{lem:conreg}
 Assume $G$ is connected, and let $u\in G$ be unipotent. 
Then:
\begin{itemize}
\item[(i)] $u$ is regular if and only if $u$ is contained in a unique Borel subgroup $B$ of $G$; 
\item[(ii)] if $u$ is regular and $P$ is a parabolic subgroup of $G$ with $u\in R_u(P)$, then $P=B$.
\end{itemize}
Similarly, any regular nilpotent element $e\in \Lie(G)$ is contained in a unique Borel subalgebra $\Lie(B)$, 
and if $P$ is a parabolic subgroup such that $e\in \Lie(R_u(P))$, then $P=B$.
\end{lem}

\begin{proof}
Part (i) is \cite[Sec.~3.7, Thm.~1]{steinberg:conjugacy}, and the analogue for the Lie algebra is \cite[Cor.~6.8]{Jantzen}.
If $P$ is a parabolic subgroup containing $u$, then $P$ contains $B$, 
and with respect to a suitable choice of maximal torus $T$ of $B$, we may write
$u = \prod_{\alpha\in \Phi^+} x_\alpha$, 
where each $x_\alpha \in X_\alpha$ and $x_\alpha \neq 1$ for each simple root $\alpha$, cf.~\cite[Sec.~3.7, Thm.~1]{steinberg:conjugacy}.
Since $u$ has a non-trivial contribution from each simple root group, $u$ can only lie in $R_u(P)$ if $P=B$.
The analogous argument works for $e$, which has a standard form involving a non-trivial contribution from 
each root space $\Lie(X_\alpha)$ relative to any simple root $\alpha$, cf.~\cite[6.7(1)]{Jantzen}.
\end{proof}

\begin{thm}\label{thm:connected}
Let $H \subseteq G$ be connected reductive groups. 
If $H$ contains a regular unipotent element of $G$, or $\Lie(H)$ contains a regular nilpotent element of $\Lie(G)$, 
then $H$ is $G$-irreducible.
\end{thm}

\begin{proof}
Suppose $u \in H$ is a regular unipotent element of $G$.
Let $B$ be a Borel subgroup of $H$ containing $u$, let $S$ be a maximal torus of $B$, and write $B = P_\lambda(H)$ for some $\lambda\in Y(S)$.
Then $u$ belongs to $R_u(B)$, so $\lim_{t\to 0} \lambda(t)u\lambda(t)^{-1} = 1$, so $u\in R_u(P_\lambda(H)) \subseteq R_u(P_\lambda)$.  It follows from Lemma~\ref{lem:conreg}(ii) that $P_\lambda$ is the unique Borel subgroup of $G$ containing $u$.

Now let $B^- = P_{-\lambda}(H)$ be the opposite Borel subgroup of $H$ with respect to the maximal torus $S$ of $H$.
The Borel subgroups $B$ and $B^-$ of $H$ are conjugate, say by $x\in H$.
Let $v = xux^{-1}\in P_{-\lambda}(H)\subseteq H$.
Since $v$ is $H$-conjugate to $u$, $v$ is also a regular unipotent element of $G$ belonging to $H$.
The argument of the first paragraph shows that $P_{-\lambda}$ is the unique Borel subgroup of $G$ containing $v$.

Now suppose $P$ is a parabolic subgroup of $G$ containing $H$.
Then $P$ contains $u$ and $v$, and hence must contain Borel subgroups normalized by $u$ and $v$, by Lemma~\ref{lem:fixedborel}.
But a Borel subgroup of $P$ is a Borel subgroup of $G$, so uniqueness forces $P$ to contain the opposite Borel subgroups $P_\lambda$ and $P_{-\lambda}$ of $G$.
This implies that $P=G$, so $H$ is $G$-ir, as required.

The proof in the case that $\Lie(H)$ contains a regular nilpotent element of $\Lie(G)$ is essentially the same --- 
given a parabolic subgroup $P$ of $G$ containing $H$, $\Lie(P)$ must contain a pair of opposite Borel subalgebras of $\Lie(G)$,
and therefore $\Lie(P) = \Lie(G)$, which means that $P=G$.
\end{proof}

\begin{rem}
	\label{rem:reg}
Note that it follows from Lemma~\ref{lem:conreg}(i) that if $H$ is a connected reductive subgroup of $G$ and $u\in H$ is a regular unipotent element of $G$, then $u$ is a regular unipotent element of $H$.
To see this, let $B$ be a Borel subgroup of $H$ containing $u$ and let $B'$ be a Borel subgroup of $G$ containing $B$.
Since $u\in B\subseteq B'$, $B'$ must be the unique Borel subgroup of $G$ containing $u$. 
Maximality of $B$ amongst connected solvable subgroups of $H$ implies that $B = (B'\cap H)^0$ is the only Borel subgroup of $H$ containing $u$, 
and we're done.
See Lemma~\ref{lem:regsub} below for this result in full generality.
\end{rem}

\begin{proof}[Proof of Theorem~\ref{thm:finite}]
	By Remark~\ref{rem:reg}, if $H$ contains a regular unipotent element of $G$, then the regular unipotent elements of $H$ are the regular unipotent elements of $G$ contained in $H$, since these elements form a single $H$-conjugacy class in $H$.  It follows from \cite[III.1.19]{springersteinberg} applied to $H$ that we may find a regular unipotent element $u$ of $G$ lying in $H_\sigma$. 
	Since $u$ is fixed by $\sigma$, the unique Borel subgroup $B$ of $H$ containing $u$ is $\sigma$-stable. 
	By Lemma~\ref{lem:fixedboreltorus}(i), there is a $\sigma$-stable maximal torus $S$ in $B$, and the opposite Borel subgroup $B^-$ to $B$ in $H$ with respect to $S$ is also $\sigma$-stable. Thanks to Lemma~\ref{lem:fixedboreltorus}(ii), $B$ and $B^-$ are conjugate by an element $x \in H_\sigma$. Thus $v = x u x^{-1}$ is a regular unipotent element of $G$ which belongs to $B^-$ and $H_\sigma$. The rest of the proof of Theorem~\ref{thm:connected} now goes through for $H_\sigma$.
	\end{proof}

\section{Regular unipotent elements}
\label{sec:regunipotent}

We collect some results about unipotent elements in non-connected reductive groups from \cite{spaltenstein};
many of these are the analogues of more familiar results for connected reductive groups. 

Following Spaltenstein \cite{spaltenstein}, we say a connected component $X$ of $G$ is \emph{unipotent} if it contains a unipotent element.  
Let $X$ be a unipotent component of $G$.  
Spaltenstein showed there is a unique unipotent $G^0$-conjugacy class $C$ in $X$ such that $C$ is dense in the set of all unipotent elements of $X$ \cite[I.4.8]{spaltenstein}.  
We call elements of $C$ \emph{regular unipotent elements} of $X$; this agrees with the usual notion if $G = G^0$. 
We say that $u\in G$ is \emph{regular unipotent} if $u$ is a regular unipotent element of some unipotent component $X$ of $G$.

\begin{ex}
A complete classification of unipotent classes when $G^0$ is simple can be found in \cite{spaltenstein}.
The essential case to consider is when the Dynkin diagram has an automorphism of order $p$.
For example, let $p=3$ and suppose $G = \langle x, G^0\rangle$, where $G^0$ has type $D_4$ and $x$ is the triality automorphism.
Then the regular unipotent elements in the component $X=xG^0$ 
are all $G^0$-conjugate to the element $xx_\alpha(1) x_\delta(1)$, where $\delta$ is the simple root corresponding to the central node on the Dynkin diagram, $\alpha$ is one of the other simple roots, and $x_\alpha$ and $x_\delta$ are the corresponding root group homomorphisms,
see \cite[I.3.1]{spaltenstein}.
\end{ex}

An element $x$ of $G$ is called \emph{quasisemisimple} if there exist a Borel subgroup $B$ of $G$ and a maximal torus $T$ of $G$ such that $x$ normalizes both $B$ and $T$ \cite[I.1.1]{spaltenstein}.  
This notion was introduced by Steinberg in case $G$ is connected \cite[Sec.~9]{steinberg:end}.  
Spaltenstein shows that any unipotent component $X$ of $G$ contains a unique $G^0$-class of quasisemisimple unipotent elements \cite[Cor.~II.2.21]{spaltenstein}, and
in fact the quasisemisimple unipotent elements in $X$ form the unique closed $G^0$-orbit in the set
of all unipotent elements in $X$ \cite[Cor.~II.2.22]{spaltenstein}.  
We give an alternative construction which works for arbitrary elements of $G$ using the machinery of $G$-complete reducibility; the
link here is that for any element $x\in G$, the $G^0$-conjugacy class of $x$ is closed if and only if the subgroup of $G$ generated by $x$ is $G$-cr, cf.~\cite[Cor.\ 3.7, Sec.~6]{BMR}. 

\begin{lem}\label{lem:quasiclass}
Let $g\in G$ and let $X$ be the component of $G$ containing $g$.
\begin{itemize}
\item[(i)] There is a unique closed $G^0$-conjugacy class in $\overline{G^0\cdot g}$, and this is a $G^0$-conjugacy class of
quasisemisimple elements in $X$.
\item[(ii)] If, in addition, $g$ is unipotent, then this quasisemisimple class is the unique closed $G^0$-orbit of unipotent elements in $X$.
\end{itemize}
\end{lem}

\begin{proof}
(i). First, the uniqueness is a standard property of orbits of reductive algebraic groups ---
for any $G^0$-action on an affine variety, 
there is a unique closed $G^0$-orbit in the closure of any $G^0$-orbit.

Let $P$ be a minimal R-parabolic subgroup of $G$ containing $g$.
There exists a Borel subgroup $B$ of $P$ normalized by $g$, by 
Lemma~\ref{lem:fixedborel},
and $N_G(B)$ is an R-parabolic subgroup of $G$ containing $g$ by Lemma~\ref{lem:N_G(P)}.
Since $B\subseteq P^0$, we have $R_u(P)\subseteq R_u(B)$, and \cite[Cor.\ 6.9]{BMR} shows that $P\cap N_G(B)$ is an 
R-parabolic subgroup of $G$ containing $g$. 
But this means that $P\subseteq N_G(B)$ by the minimality of $P$, and hence $P^0 = B$.
Let $T$ be a maximal torus of $P$, let $L$ be the R-Levi subgroup of $P$ with $L^0=T$, and let $\lambda\in Y(G)$ be such that $P=P_\lambda$ and $L=L_\lambda$.
It follows from \cite[Ex.\ 4.8]{BMR:semisimplification} that
$x:=\lim_{t\to 0}\lambda(t)g\lambda(t)^{-1} \in L$ generates a 
$G$-cr subgroup of $G$, and hence the $G$-orbit of $x$ is closed by \cite[Cor.\ 3.7, Sec.\ 6]{BMR}.
Since $x\in P$, $x$ normalizes $P^0 = B$; since $x\in L$, $x$ normalizes $L^0=T$; since the $G$-conjugacy class of $x$ is closed, so is the $G^0$-conjugacy class; since $x$ is obtained as a limit from $g$ along a cocharacter which evaluates in $G^0$, $x\in\overline{G^0\cdot g}$; since $X$ is a closed subset of $G$, we also have $x\in X$.  Moreover, if $g$ is unipotent then $x$ is unipotent, since the set of unipotent elements is closed in $G$.

(ii). Note that since the class of regular unipotent elements in $X$ is dense in the set of all unipotent elements \cite[I.4.8]{spaltenstein}, it follows that there is only one closed $G^0$-orbit of unipotent elements in $X$, and it must be the one constructed in the first paragraph for any unipotent $g\in X$.
\end{proof}

Spaltenstein also proves the following \cite[Prop.~II.10.2]{spaltenstein}, which is the crucial ingredient in the proof of Theorem~\ref{thm:main}.

\begin{prop}
\label{prop:regunipcrit}
 Let $u\in G$ be unipotent.  Then $u$ is regular unipotent if and only if $u$ normalizes a unique Borel subgroup of $G$.
\end{prop}

We quickly obtain the following, which is also used in the proof of the main theorem.

\begin{lem}\label{lem:containsB}
Let $P$ be an R-parabolic subgroup of $G$ containing a regular unipotent element $u$ of $G$.
Then $P^0$ contains the unique Borel subgroup of $G$ normalized by $u$.  
\end{lem}

\begin{proof}
Given that $u\in P$, $u$ normalizes a Borel subgroup $B$ of $P$, by Lemma~\ref{lem:fixedborel}.
But a Borel subgroup of $P$ is also a Borel subgroup of $G$, and so $B$ is the unique Borel subgroup of $G$
normalized by $u$ given by Proposition~\ref{prop:regunipcrit}.
Since $B$ is connected by definition, $B\subseteq P^0$.
\end{proof}

\begin{rem}
\label{rem:isog}
 It follows from Proposition~\ref{prop:regunipcrit} that if $f\colon G_1\to G_2$ is an isogeny of reductive groups and $u\in G_1$ is unipotent then $u$ is regular unipotent in $G_1$ if and only if $f(u)$ is regular unipotent in $G_2$.  This is because a subgroup $B$ of $G_1^0$ is a Borel subgroup of $G_1^0$ if and only if $f(B)$ is a Borel subgroup of $G_2^0$.
\end{rem}

\begin{rem}
	\label{rem:distinguished}
Spaltenstein also introduces the notion of a \emph{distinguished unipotent element} of $G$, \cite[II.3.13]{spaltenstein}; that is, a unipotent element $u$ such that every torus in $C_G(u)$ is central in $G^0$.
It follows from \cite[Prop.\ II.3.16, Prop.\ II.10.2]{spaltenstein} that a regular unipotent element in $G$ is distinguished. For $G$ connected this notion is due to Bala--Carter, cf.~\cite[Sec.~5]{carter:book}.
\end{rem}

\begin{cor}
\label{cor:GcrGir}
Let $H$ be a $G$-completely reducible subgroup of $G$ containing a regular unipotent element $u$ of $G$.  Then $H$ does not normalize any proper parabolic subgroup of $G^0$.
\end{cor}

\begin{proof}
 Suppose $H$ normalizes a parabolic subgroup $P$ of $G^0$.  Then $H\subseteq N_G(P)$, which is an R-parabolic subgroup of $G$ by Lemma~\ref{lem:N_G(P)}.  By hypothesis, $H$ is contained in an R-Levi subgroup $L$ of $N_G(P)$.  Choose  $\lambda\in Y(G)$ such that $N_G(P)= P_\lambda$ and $L= L_\lambda$.  Since $\lambda$ centralizes $u$, $\lambda$ must belong to $Y(Z(G^0))$, by 
 Remark~\ref{rem:distinguished}. It follows that $L_\lambda^0= G^0$, which implies that $P= G^0$.
\end{proof}

We finish the section by showing that the notion of a regular unipotent element behaves nicely when we pass to quotients and reductive subgroups of $G$.

\begin{lem}
\label{lem:regsub}
 Let $u$ be a regular unipotent element of $G$.  Let $H$ be a reductive subgroup of $G$ such that $u\in H$.  Then $u$ is a regular unipotent element of $H$.
\end{lem}

\begin{proof}
 It is enough by Proposition~\ref{prop:regunipcrit} to show that $u$ normalizes a unique Borel subgroup of $H$.  By Lemma~\ref{lem:fixedborel}, $u$ normalizes at least one Borel subgroup of $G$.  Suppose $B_1$ and $B_2$ are Borel subgroups of $H$ normalized by $u$, and let $h\in H^0$ be such that $B_2 = hB_1h^{-1}$.
Then $N_H(B_i)$ is an R-parabolic subgroup of $H$ containing $u$, with $N_H(B_i)^0 = B_i$, for $i=1,2$, by Lemma~\ref{lem:N_G(P)}.
Note also that $N_H(B_2) = hN_H(B_1)h^{-1}$.
Thus we may find a cocharacter $\lambda\in Y(H)$ with $N_H(B_1) = P_\lambda(H)$ and $N_H(B_2) = hP_\lambda(H)h^{-1} = P_{h\cdot\lambda}(H)$.

Now $P_\lambda^0$ and $P_{h\cdot\lambda}^0$ are parabolic subgroups of $G^0$ normalized by $u$, 
and hence $P_\lambda^0$ and $P_{h\cdot\lambda}^0$ both contain 
the unique Borel subgroup of $G$ normalized by $u$, by Lemma~\ref{lem:containsB}.
But conjugate parabolic subgroups of $G^0$ containing a common Borel subgroup are equal, so $P_\lambda^0 = P_{h\cdot\lambda}^0 = hP_\lambda^0h^{-1}$.
Since $h\in H^0\subseteq G^0$ normalizes the parabolic subgroup $P_\lambda^0$, we have $h\in P_\lambda^0\cap H^0 = P_\lambda(H^0) = P_\lambda(H)^0 = B_1$.
We finally conclude that $B_1=B_2$, as required.
\end{proof}

The special case of Lemma~\ref{lem:regsub} when $G$ is simple and $H$ is connected is \cite[Lem.~2.10]{MT}. 

\begin{lem}
\label{lem:regquot}
 Let $u$ be a regular unipotent element of $G$.  Let $G'$ be a quotient of $G$ and let $\pi\colon G\to G'$ be the canonical projection.  Then $\pi(u)$ is a regular unipotent element of $G'$.
\end{lem}

\begin{proof}
 Let $N= \ker \pi$ (set-theoretic kernel).  The canonical projection factors as $G\to G/N^0\to (G/N^0)/(N/N^0)$, so we can assume without loss by Remark~\ref{rem:isog} that $N= N^0$.  By \cite[Lem.~2.6]{BMR:commuting}, there exists a subgroup $M$ of $G$ such that $MN= G$, $M\cap N$ is a finite normal subgroup of $M$, $M^0\cap N$ is central in both $M^0$ and $N^0$, and $M^0$ commutes with $N$; in particular, $M$ is normal in $G$.  
 By Remark~\ref{rem:isog}, we can assume that $G= M\ltimes N$, $G'= M$ and $G^0= M^0\times N^0$.
 
 Let $B_1$ and $B_2$ be Borel subgroups of $G/N$ normalized by $\pi(u)$.  
 By the previous paragraph, we may regard $B_1$ and $B_2$ as subgroups of $M$ normalized by the conjugation action of $u$ on $M$.  
 There is also a Borel subgroup $B$ of $N$ normalized by the action of $u$ on $N$, by Lemma~\ref{lem:fixedborel}.  
 Clearly $BB_1$ and $BB_2$ are Borel subgroups of $G$ normalized by $u$.  
 Since $u$ is regular unipotent in $G$, $BB_1= BB_2$ by Proposition~\ref{prop:regunipcrit}.  
 Hence $B_1= B_2$.  
 Another application of Proposition~\ref{prop:regunipcrit} and Lemma~\ref{lem:fixedborel} gives that $\pi(u)$ is regular unipotent in $G/N$, as required.
\end{proof}

\section{Proof of Theorem~\ref{thm:main}}
\label{sec:proof}

The proof of Theorem~\ref{thm:main} follows the same lines as the connected case Theorem~\ref{thm:connected}.  
The crucial point is to show that any R-parabolic subgroup of $G$ containing $H$
must contain the unipotent radicals of a pair of opposite parabolic subgroups of $G^0$.
Our next results indicate how this allows us to deduce part (iii) of Theorem~\ref{thm:main}.

\begin{lem}
\label{lem:uniprad}
Let $H$ be a reductive subgroup of $G$.  
Let $u\in H$ be a regular unipotent element of $G$ and let $Q$ be an R-parabolic subgroup
of $G$ containing $u$.
Then every R-parabolic subgroup of $G$ containing $H$ also contains $R_u(Q)$.
\end{lem}

\begin{proof}
Let $B$ be the unique Borel subgroup of $G^0$ normalized by $u$.  
Then $B\subseteq Q^0$ by Lemma~\ref{lem:containsB}, so $R_u(Q^0) = R_u(Q)\subseteq R_u(B)$.  
Now suppose $P$ is an R-parabolic subgroup of $G$ containing $H$.  
Then $u\in P$, so $B\subseteq P$, again by Lemma~\ref{lem:containsB}, so $R_u(Q)\subseteq P$,
as required.  
\end{proof}

\begin{prop}
\label{prop:opp}
 Let $H$ be a reductive subgroup of $G$ which meets every connected component of $G$.  
 Let $u_1,u_2\in H$ be regular unipotent elements of $G$.  
 Suppose there are R-parabolic subgroups $P_1,P_2$ of $G$ such that: 
 \begin{itemize}
 \item[(i)] $P_1$ and $P_2$ do not contain any simple component of $G^0$;
 \item[(ii)] $u_1\in P_1$, $u_2\in P_2$;
 \item[(iii)] $P_1^0$ and $P_2^0$ are opposite parabolic subgroups of $G^0$.
 \end{itemize}  
 Then $H$ is $G$-irreducible.
\end{prop}

\begin{proof}
Suppose $P$ is an R-parabolic subgroup of $G$ containing $H$.  
Then $u_1, u_2\in P$, so $R_u(P_1)\cup R_u(P_2)\subseteq P$, by Lemma \ref{lem:uniprad}.
But $P_1^0$ and $P_2^0$ are opposite parabolic subgroups of $G^0$ that do not contain any simple component of $G^0$, so $R_u(P_1^0)\cup R_u(P_2^0)$ generates $[G^0,G^0]$, by Lemma~\ref{lem:opp_rads}.  
Hence $P\supseteq G^0$.  
Since $H$ meets every connected component of $G$, we therefore have $P=G$.  
This shows that $H$ is $G$-ir, as claimed.
\end{proof}

Armed with these results we now address the main theorem.

\begin{proof}[Proof of Theorem~\ref{thm:main}]
We first note that if $H^0$ is a torus, then (i) holds automatically, cf.~\cite[Prop.~3.20; Sec.~6.3]{BMR}, and (ii) and (iii) are not relevant, so we may assume that $H^0$ is not a torus for the remainder of the proof. This means in particular that the Borel subgroups in $H^0$ are proper.

We begin with some general observations.
Let $u_1\in H$ be a regular unipotent element of $G$. 
Then $u_1$ is also regular in $H$, by Lemma~\ref{lem:regsub},
and hence there is a unique Borel subgroup $B$ of $H$ normalized by $u_1$, by Proposition~\ref{prop:regunipcrit}.
By Lemma~\ref{lem:N_G(P)}, $N_H(B)$ is an R-parabolic subgroup of $H$, so we may choose a maximal torus $S$ of $B$ and a cocharacter $\lambda\in Y(S)$ with $N_H(B) = P_\lambda(H)$.
Let $B^-$ be the opposite Borel subgroup of $H^0$ such that $B\cap B^- = S$.
We claim that $P_{-\lambda}(H) = N_H(B^-)$. 
It is easy to see that $P_{-\lambda}(H)^0 = P_{-\lambda}(H^0) = B^-$, so $P_{-\lambda}(H)\subseteq N_H(B^-)$.
Now note that $P_\lambda$ and $P_{-\lambda}$ have the same number of components, since $L_\lambda = L_{-\lambda}$.
Further, since $B$ and $B^-$ are conjugate by an element $x\in N_{H^0}(S)$, the normalizers $N_H(B)$ and $N_H(B^-)$ are conjugate by $x$ too; this implies that $N_H(B^-)$ has the same number of components as $N_H(B)$.
Thus $N_H(B^-)$ and $P_{-\lambda}(H)$ have the same number of components, and we have proved the claim.
By setting $u_2 := xu_1x^{-1}\in P_{-\lambda}(H)$ we obtain another regular unipotent element of $G$ and $H$; note that $B^-$ is the unique Borel subgroup of $H$ normalized by $u_2$.

We can now prove part (i).
We argue by contradiction. Suppose $H^0$ is not $G$-cr.
Then $[H^0,H^0]$ is not $G$-cr either, by Lemma~\ref{lem:derivedcr},
and so we may apply Theorem~\ref{thm:optpar} to $[H^0,H^0]$ and let $P$ be an R-parabolic subgroup of $G$ with the properties given there.
Since $H$ normalizes $[H^0,H^0]$, we have $H\subseteq P$ by Theorem~\ref{thm:optpar}(ii).  
Keeping the notation from the previous paragraph, let $P_1:= P_\lambda$ and $P_2 := P_{-\lambda}$;
then Lemma \ref{lem:uniprad} implies that $R_u(P_1)$ and $R_u(P_2)$ are contained in $P$,
and hence the subgroup $M$ generated by $R_u(P_1)$ and $R_u(P_2)$ is contained in $P$.
Since $M$ is $G$-cr by Lemma~\ref{lem:opp_rads}, there is an R-Levi subgroup $L$ of $P$ containing $M$.
But $M$ contains $R_u(P_1)\cap H = R_u(P_\lambda(H))$ and $R_u(P_2)\cap H = R_u(P_{-\lambda}(H))$, which are the unipotent radicals of opposite Borel subgroups of $H$.
Thus $M$ contains all the root groups of $H^0$ with respect to the maximal torus $S$, and hence $[H^0,H^0]\subseteq M\subseteq L$.
This contradicts Theorem~\ref{thm:optpar}(i), and this contradiction completes the proof.

Now we prove (ii) and (iii).
To do so, we may replace $G$ with the subgroup $\widetilde{G}$ generated by $G^0$ and $H$, since  (iii) holds for $H$ in $\widetilde{G}$ if and only if (ii) holds for $H$ in $G$.
Thus we may also assume that $H$ meets every component of $G$.
Note that $P_1= P_\lambda$ and $P_2 = P_{-\lambda}$ satisfy hypotheses (ii) and (iii) of Proposition~\ref{prop:opp}.
Suppose that, in addition, the projection of $H^0$ to each simple factor of $G$ is not a torus.  
It follows from Lemma~\ref{lem:torimage} applied to each of these projection maps that $P_1$ and $P_2$ do not contain any simple factor of $G^0$.  Hence hypothesis (i) of Proposition~\ref{prop:opp} holds, so we may conclude from Proposition~\ref{prop:opp} that $H$ is $G$-ir.  This completes the proof of Theorem~\ref{thm:main}.
\end{proof}

As noted in the introduction, Theorem~\ref{thm:main}(i) holds without the more restrictive hypotheses needed for parts (ii) and (iii); this allows us to give the following interesting corollary.

\begin{cor}\label{cor:connectedcr}
Suppose $H$ is a connected reductive subgroup of $G$. If $H$ is normalized by a regular unipotent element of $G$, then $H$ is $G$-completely reducible.
\end{cor}

\begin{proof}
Let $u$ be a regular unipotent element of $G$ normalizing $H$.
Since reductivity and complete reducibility are equivalent in characteristic $0$ (see \cite[Prop.~4.2]{serre2}, \cite[Sec.~2.2, Sec.~6.3]{BMR}), we may assume that $u$ has finite order.
Let $K$ be the subgroup of $G$ generated by $u$ and $H$; then $K^0 = H$, so $K$ is reductive.
Since $K$ contains $u$, Theorem~\ref{thm:main}(i) applied to $K$ gives the result.
\end{proof}

We finish this section by proving the analogue of Theorem~\ref{thm:main} in the presence of a Steinberg endomorphism $\sigma$ of $G$, generalizing Theorem~\ref{thm:finite}.

\begin{prop}
	\label{prop:steinbergendgeneral}
Let $H \subseteq G$ be reductive algebraic groups (possibly non-connected). 
Let $\sigma$ 
be a Steinberg endomorphism of $G$ 
with $\sigma(H)\subseteq H$.  Suppose the projection of $H^0$ onto each simple factor of $G^0$ is not a torus.
If some $\sigma$-stable connected component $X$ of $H$ contains a regular unipotent element of $G$, then $H_\sigma$ does not normalize any proper parabolic subgroup of $G^0$.
If, moreover, $H_\sigma$ meets every connected component of $G$, then $H_\sigma$ is $G$-irreducible.
\end{prop}

\begin{proof}
Suppose $X$ is a $\sigma$-stable connected component of $H$ containing a regular unipotent element $u$ of $G$.  Then $u$ is a regular unipotent element of $H$ by Lemma~\ref{lem:regsub}.  Hence the regular unipotent elements of $H$ in $X$ are the regular unipotent elements of $G$ belonging to $X$, since these elements form a single $H^0$-conjugacy class $C$ in $X$, \cite[II.10.1]{spaltenstein}.  By \cite[I.2.7]{springersteinberg}, $C$ contains a $\sigma$-fixed point $u_1$.

We now follow the proof of Theorem~\ref{thm:main} above, taking $u_1\in H_\sigma$.  Since $\sigma(u_1) = u_1$ and $B$ is the unique Borel subgroup of $H$ normalized by $u_1$, $B$ must be $\sigma$-stable.
Hence we may choose $S$ and $B^-$ in the proof to be $\sigma$-stable as well, and
the element $x\in N_{H^0}(S)$ conjugating $B$ to $B^-$ can be chosen in $H_\sigma$, by Lemma~\ref{lem:fixedboreltorus}.
So $u_2\in H_\sigma$ as well.
Now the rest of the proof goes through unchanged: we conclude that $H_\sigma$ does not normalize any proper parabolic subgroup of $G^0$, 
and if $H_\sigma$ meets every component of $G$ then $H_\sigma$ is $G$-ir.
\end{proof}

\section{Further discussion}
\label{sec:complements}

We finish the paper with a discussion of some extensions to the main result and some examples to illustrate other points of interest.  We start by exploring the limits on the hypotheses on $H$ and $H^0$ in Theorem~\ref{thm:main}. 

\begin{rem}
	\label{rem:char0}
Recall that in characteristic $0$ a subgroup is $G$-cr if and only if it is reductive, \cite[Prop.\ 4.2]{serre2}, \cite[Sec.~2.2, Sec.~6.3]{BMR}.
Hence Theorem~\ref{thm:main} follows quickly in characteristic $0$ from Corollary~\ref{cor:GcrGir}.
There is a similar equivalence between complete reducibility and reductivity in positive characteristic $p$ if the index of $H^0$ in $H$ is coprime to $p$ and 
$p$ is sufficiently large relative to the rank of $G$, see \cite[Thm.\ 4.4]{serre2}.
Thus the conclusion of Theorem~\ref{thm:main}(i) is particularly interesting in ``small characteristics'', where it doesn't simply follow from the reductivity of $H$,
and the conclusions of parts (ii) and (iii) are of particular note when the index of $H^0$ in $H$ is divisible by $p$.
\end{rem}

\begin{rem}
	\label{rem:maxtorus}
As noted in Remark~\ref{rem:main}(iv), some restrictions on $H^0$ are necessary for Theorem~\ref{thm:main} to hold.
When $G^0$ is simple, non-$G$-irreducible examples of subgroups $H$ of $G$ containing a regular unipotent element must have $H^0$ a torus; we refer to \cite[Sec.~7]{MT} for such examples. 

On the other hand, we also note that there are plenty of examples where $H^0$ is a torus and yet the conclusions of Theorem~\ref{thm:main} \emph{do} hold.
For example, suppose $H^0$ is a regular torus of $G$ (i.e., one containing a regular semisimple element of $G$, \cite[IV.13.1]{borel}), so that $T = C_G(H^0)^0$ is a maximal torus of $G$.  Then $N_G(H^0)^0= C_G(H^0)^0 = T$, so $N_G(H^0)$ is a finite extension of $T$ and every subgroup of $N_G(H^0)$ is $N_G(H^0)$-cr; in particular, $H$ is $N_G(H^0)$-cr.  Since $H^0$ is $G$-cr by \cite[Prop.~3.20, Sec.~6.3]{BMR}, it follows from 
\cite[Cor.~3.3]{BMR:commuting} applied to the inclusion of $H^0$ in $H$
that $H$ is $G$-cr.  Hence by Corollary~\ref{cor:GcrGir} the conclusions of Theorem~\ref{thm:main}(ii) and (iii) hold in this case even though the hypotheses do not.
For a concrete example of this phenomenon, let $H$ be the normalizer of the maximal torus in $G = \SL_2$ when the characteristic is $2$: then $H$ contains 
a regular unipotent element and is $G$-ir.
\end{rem}

\begin{rem}
 Let $H \subseteq G$ be reductive algebraic groups (possibly non-connected) and suppose $H$ contains a regular unipotent element of $G$.  If there is a nontrivial normal unipotent subgroup $N$ of $H$ such that $N\subseteq G^0$ then the conclusions of Theorem~\ref{thm:main} fail for $H$: for $N$ is not $G^0$-cr, so $N$ is not $G$-cr, so $H$ is not $G$-cr \cite[Thm.~3.10]{BMR}.  The ${\rm PSL}_2(p)$ examples we discussed after Theorem~\ref{thm:finite} show, however, that the conclusions of Theorem~\ref{thm:main} can fail even when $H$ is finite simple.
\end{rem}

\begin{rem}
For simple $G$ and semisimple $H$, the classification in \cite[Thm. 1.4]{TZ} says that apart from some obvious classical group cases, such as $\SP_{2n}\subseteq \SL_{2n}$,
pairs $H\subseteq G$ with $H$ containing a regular unipotent element of $G$ are rare.
However, if $G$ is not simple and $G$ and $H$ are allowed to be disconnected, 
then it is much harder to keep track of the possibilities in a systematic way; see also our next example.
\end{rem}

\begin{ex}\label{ex:infinitelymany}
 The following example shows that a given regular unipotent element of $G$ can belong to infinitely many distinct connected reductive overgroups.  Let $G= \SL_2\times \SL_2$ and assume ${\rm char}(k)= p> 0$.  Fix unipotent $1\neq v\in \SL_2({\mathbb F}_p)$ and let $u= (v,v)$.  For each power $q$ of $p$, define $H_q$ to be the image of $\SL_2$ under the twisted Frobenius diagonal embedding $g\mapsto (g, \sigma_q(g))$, where $\sigma_q$ is the standard $q$-power Frobenius map.  Then the $H_q$ are distinct --- in fact, they are pairwise nonconjugate --- and $u\in H_q$ for all $q$.
\end{ex}


\bigskip
\bigskip
\noindent {\bf Acknowledgements}:
We thank Donna Testerman for comments on an earlier version of the paper, and the referee for several suggestions which have improved the exposition.


\end{document}